\RequirePackage[l2tabu, orthodox]{nag}
\documentclass[a4paper,draft,11pt]{article}

\usepackage[a4paper]{geometry}

\usepackage{pgf,tikz,bm}
\usetikzlibrary{arrows}

\usepackage[final]{microtype}
\makeatletter
\def\MT@register@subst@font{\MT@exp@one@n\MT@in@clist\font@name\MT@font@list
   \ifMT@inlist@\else\xdef\MT@font@list{\MT@font@list\font@name,}\fi}
\makeatother 

\usepackage[small]{titlesec}

\usepackage{amsmath, amssymb, amsfonts, amsthm, mathtools}
\usepackage{fixmath}
\mathtoolsset{centercolon}

\usepackage[latin1]{inputenc}
\usepackage[T1]{fontenc}

\usepackage[draft=false,pdftex]{hyperref}
\hypersetup{
    bookmarks=true,         
    unicode=false,          
    pdftoolbar=true,        
    pdfmenubar=true,        
    pdffitwindow=false,     
    pdfstartview={FitH},    
    pdftitle={Favourite distances in 3-space},    
    pdfauthor={Konrad J. Swanepoel},     
    pdfnewwindow=false,      
    colorlinks=true,       
    linkcolor=black,          
    citecolor=black,        
    filecolor=black,      
    urlcolor=black           
}

\theoremstyle{plain}
\newtheorem{theorem}{Theorem}

\newtheorem{corollary}[theorem]{Corollary}
\newtheorem{proposition}[theorem]{Proposition}
\newtheorem{mytheorem}{Theorem}

\newtheorem*{KST1}{K\H{o}vari-S\'os-Tur\'an Theorem I}
\newtheorem*{KST2}{K\H{o}vari-S\'os-Tur\'an Theorem II}

\theoremstyle{definition}

\newcommand{\setbuilder}[2]{\left\{#1\;\colon\,#2\right\}}
\newcommand{\set}[1]{\left\{#1\right\}}
\newcommand{\epsi}{\varepsilon}
\newcommand{\fhi}{\varphi}
\newcommand{\abs}[1]{\left\lvert#1\right\rvert}
\newcommand{\myangle}{\sphericalangle}
\newcommand{\length}[1]{\lvert#1\rvert}
\newcommand{\card}[1]{\lvert#1\rvert}

\newcommand{\numbersystem}[1]{\mathbb{#1}}
\newcommand{\bN}{\numbersystem{N}}
\newcommand{\bR}{\numbersystem{R}}
\newcommand{\bZ}{\numbersystem{Z}}
\newcommand{\dE}{\vec{E}}
\newcommand{\dG}{\vec{G}}
\newcommand{\dK}{\vec{K}}
\newcommand{\collection}[1]{{\mathcal#1}}
\newcommand{\CC}{\collection{C}}
\newcommand{\CL}{\collection{L}}
\newcommand{\vect}[1]{\bm{#1}}

\newcommand{\vx}{\vect{x}}
\newcommand{\vy}{\vect{y}}

\newcommand{\define}[1]{\emph{#1}}

\title{Favourite distances in 3-space}
\author{Konrad J.\ Swanepoel\footnote{Department of Mathematics, London School of Economics and Political Science, United Kingdom. Email: \href{mailto:k.swanepoel@lse.ac.uk}{k.swanepoel@lse.ac.uk}}}
\date{}

\begin{document}
\maketitle

\begin{abstract}
Let $S$ be a set of $n$ points in Euclidean $3$-space.
Assign to each $x\in S$ a distance $r(x)>0$, and let $e_r(x,S)$ denote the number of points in $S$ at distance $r(x)$ from $x$.
Avis, Erd\H{o}s and Pach (1988) introduced the extremal quantity $f_3(n)=\max\sum_{x\in S}e_r(x,S)$, where the maximum is taken over all $n$-point subsets $S$ of $3$-space and all assignments $r\colon S\to(0,\infty)$ of distances.
We show that if the pair $(S,r)$ maximises $f_3(n)$ and $n$ is sufficiently large, then, except for at most $2$ points, $S$ is contained in a circle $\CC$ and the axis of symmetry $\CL$ of $\CC$, and $r(x)$ equals the distance from $x$ to $C$ for each $x\in S\cap\CL$.
This, together with a new construction, implies that $f_3(n)=n^2/4 + 5n/2 + O(1)$.
\end{abstract}

\section{Introduction}
Let $S$ be a set of $n$ points in the $d$-dimensional Euclidean space $\bR^d$.
We write $d(x,y)$ for the Euclidean distance between $x$ and $y$, and $d(x,A)=\min\setbuilder{d(x,a)}{a\in A}$ for the distance from $x$ to the finite set $A\subset\bR^d$.
Let $r\colon S\to(0,\infty)$ be a choice of a positive number for each point in $S$.
Define the \define{favourite distance digraph on $S$ determined by $r$} to be the directed graph $\dG_r(S)=(S,\dE_r(S))$ on the set $S$ with arcs
\[\dE_r(S):=\setbuilder{(x,y)}{x,y\in S\text{ and }d(x,y)=r(x)}\text{.}\]
Let $\card{A}$ denote the cardinality of the set $A$.
Define
\[ f_d(n) := \max\setbuilder{\card{\dE_r(S)}}{S\subset\bR^d, \card{S}=n \text{ and } r\colon S\to(0,\infty)}\text{.} \]

The problem of determining $f_d(n)$ was originally introduced by Avis, Erd\H{o}s and Pach \cite{AEP}, who showed that
\[\frac{n^2}{4}+\frac{3n}{2} \leq f_3(n)\leq\frac{n^2}{4}+ an^{2-b}\] for some constants $a,b>0$.
Our main result is the following asymptotic improvement:
\begin{mytheorem}\label{thm:3d}
For all sufficiently large $n$, \[\left\lceil\frac{n^2}{4}+\frac{5n}{2}\right\rceil+1\leq f_3(n)\leq \left\lceil\frac{n^2}{4}+\frac{5n}{2}\right\rceil+12.\]
\smallskip
\end{mytheorem}
The upper bound follows from the following structural result and the lower bound from a construction in Section~\ref{section:suspension}.
A finite set $S$ of points in $\bR^3$ and a function $r\colon S\to(0,\infty)$ are called a \define{suspension} if $S$ is contained in the union of some circle $\CC$ and its axis of symmetry $\CL$, and $r\colon S\to(0,\infty)$ satisfies $r(x)=d(x,\CC)$ for all $x\in S\cap\CL$.
If $\card{S}>7$ then $\CC$ and $\CL$ are uniquely determined by $S$.
The next result states that if $S$ and $r$ are extremal, then $S$ is a suspension except for at most $2$ points.
\begin{mytheorem}\label{thm:3dextremal}
Let $S\subset\bR^3$ be finite and $r\colon S\to(0,\infty)$ be such that $e_r(S)=f_3(\card{S})$.
If $\card{S}$ is sufficiently large then for some $T\subseteq S$ with $\card{T}\leq 2$, $S\setminus T$ is a suspension with circle $\CC$ and symmetry axis $\CL$.
\end{mytheorem}
We conjecture that the exceptional set $T$ is empty if $\card{S}$ is large.
We prove the above theorem using the following stability result, which states that if $(S,r)$ is almost extremal, then it is a suspension up to $o(n)$ points.
\begin{mytheorem}\label{thm:3dstability}
For each $\epsi>0$ there exists $\delta>0$ and $n_0\in\bN$ such that if $n\geq n_0$ and $S\subset\bR^3$ is a set of $n$ points with $e_r(S) > (\frac{1}{4}-\delta)n^2$, then for some $T\subset S$ with $\card{T}<\epsi n$, $S\setminus T$ is a suspension with circle $\CC$ and symmetry axis $\CL$, $\abs{\card{S\cap\CC}-n/2}<\epsi n$ and $\abs{\card{S\cap\CL}-n/2}<\epsi n$.
\end{mytheorem}
In the paper \cite{Sw-extremal4d} we determined $f_d(n)$ for all $d\geq 4$ and $n$ sufficiently large depending on $d$.
See also Erd\H{o}s and Pach \cite{EP}.
Csizmadia \cite{Cs} determined the maximum number of \emph{furthest distance pairs} in $\bR^3$, where we fix $r(\vx)= \max_{\vy\in S}\length{\vx\vy}$.
Favourite and furthest distances in the plane have been considered by Avis \cite{Avis} and Edelsbrunner and Skiena \cite{ES}.

In the next section we make a careful construction which proves the lower bound of Theorem~\ref{thm:3d}.
In Section~\ref{section:simple} we give a relatively straightforward induction proof of the upper bound $f_3(n)\leq n^2/4 + O(n^{5/3})$ (Theorem~\ref{thm:induction}).
Using the ideas of this proof we then prove the main theorems in Section~\ref{section:mainproofs}.

\section{Suspensions}\label{section:suspension}
In this section we consider some properties of favourite distance digraphs on suspensions and estimate the maximum of $\card{\dE_r(S)}$ taken over all suspensions $S$ of $n$ points.
Define $e_r(S)=\card{\dE_r(S)}$ and $e_r(A,B)=\card{\dE_r(S)\cap(A\times B)}$ for any $A,B\subseteq S$.
Without loss of generality, we let the radius of the circle $\CC$ of the suspension be $1$ and we identify $\CL$ with the real line $\bR$ so that the centre of $\CC$ is $0\in\bR$.
Write $C=S\cap\CC$ and $L=S\cap\CL$.
Then $r(x)=\sqrt{1+x^2}$ for each $x\in L$.
Consider the subdigraph $\dG_r(L)$.
For any arc $(x,y)\in\dE_r(L)$, we have $\abs{x-y}=d(x,y)=r(x)=\sqrt{1+x^2}$.
If we solve for $y$ we obtain $y=x\pm\sqrt{1+x^2}=:s^{\pm}(x)$, hence
each vertex $x\in L$ has at most two out-neighbours in $L$.
If we solve for $x$, we obtain $x=\frac12(y-1/y)=:p(y)$, hence each vertex $y\in L$ has at most one in-neighbour in $L$ (and if $0\in L$ then $0$ cannot have any in-neighbour).
Therefore, $e_r(L)\leq\card{L}:=\ell$.
Next, note that if $e_r(L)=\ell$, then each vertex has exactly one in-neighbour, and it follows that each connected component of $\dG_r(L)$ consists of a directed cycle of length at least $2$ together with binary trees where each binary tree is attached to the cycle at its root and its arcs are directed away from the root.

We can use angles to give a simple description of the dynamics of the \emph{predecessor} $p$ and \emph{successors} $s^{\pm}$ of a vertex in $L$.
For each $x\in\CL$, let $\theta(x)=\pi/2+\arctan x$.
Then $\theta(x)\in(0,\pi)$ is the angle between the ray from $x$ to a point on the circle $\CC$ and the ray from $x$ in the positive direction on $\CL$.
Then simple properties of angles in circles give that $\theta(s^+(x))=(\pi+\theta)/2$, $\theta(s^-(x))=\theta/2$ and $\theta(p(x))=2\theta(x)\pmod{\pi}$ (Figure~\ref{fig:angles}).
\begin{figure}
\centering
\begin{tikzpicture}[line cap=round,line join=round,>=triangle 45,scale=0.8, rotate=-0.5]
\clip(-2.4,0.025513075076106492) rectangle (13.973252175534737,6.427907128260514);
\fill [fill=blue!10] (0.52,1.64) ellipse [x radius = 1, y radius = 3.96];
\draw [draw=black!70, dashed] (0.52,-2.32) arc [start angle =-90, end angle = 90, x radius = 1, y radius = 3.96];
\draw [draw=black!90] (0.52,5.6) arc [start angle =90, end angle = 270, x radius = 1, y radius = 3.96];
\draw(1.0505014039484695,1.6453586000398834) -- (1.0451428039085862,2.175860003988353) -- (0.5146413999601166,2.1705014039484696) -- (0.52,1.64) -- cycle; 
\draw [shift={(-0.8637677754364076,1.6260225477228647)}] (0,0) -- (0.5787255656077527:1) arc (0.5787255656077527:71.31745699349653:1) -- cycle;
\draw [shift={(11.852537350824202,1.7544700742507484)}] (0,0) -- (0.5787255656077883:0.7502805531075489) arc (0.5787255656077883:161.31745699349653:0.7502805531075489) -- cycle;
\draw [shift={(5.494384787693898,1.6902463109868069)}] (0,0) -- (0.5787255656077537:0.8) arc (0.5787255656077537:142.0561884213853:0.8) -- cycle;
\draw [domain=0.52:13.973252175534737] plot(\x,{(--6.4736--0.04*\x)/3.96});
\draw (0.52,1.64)-- (0.48,5.6);
\draw(5.494384787693898,1.6902463109868069) circle (6.358476917297747cm);
\draw (0.48,5.6)-- (5.494384787693898,1.6902463109868069);
\draw (0.48,5.6)-- (11.852537350824202,1.7544700742507484);
\draw (-0.8637677754364076,1.6260225477228647)-- (11.852537350824202,1.7544700742507484);
\draw (-0.8637677754364076,1.6260225477228647)-- (0.48,5.6);
\draw (1.5,3.1) node[right] {\color{blue!80!black}$\CC$};
\draw (3.8,3.1) node[right] {$\sqrt{1+x^2}$};
\draw [fill=black] (0.52,1.64) circle (2pt) node[below] {$0$};
\draw (0.5,3.6) node[right] {$1$};
\draw [fill=black] (0.48,5.6) circle (2pt);
\draw [fill=black] (5.494384787693898,1.6902463109868069) circle (2pt);
\draw[color=black] (5.670147387811196,1.3) node {$x$};
\draw[color=black] (3.1,1.3) node {$\CL=\bR$};
\draw [fill=black] (-0.8637677754364076,1.6260225477228647) circle (2pt);
\draw[color=black] (-1.6,1.4) node {$s^-(x)$};
\draw [fill=black] (11.852537350824202,1.7544700742507484) circle (2pt);
\draw[color=black] (12.6,1.3) node {$s^+(x)$};
\draw[color=black] (-0.3,2.05) node {$\frac{\theta}{2}$};
\draw[color=black] (12.9,2.45) node {$\frac{\pi+\theta}{2}$};
\draw[color=black] (5.7,2.1) node {$\theta$};
\end{tikzpicture}
\caption{The mappings $s^{\pm}\colon\bR\to\bR$}\label{fig:angles}
\end{figure}
Thus, if we consider the binary expansion of $\alpha(x):=\theta(x)/\pi\in(0,1)$, then $\alpha(p(x))$ is the left-shift of $\alpha(x)$, $\alpha(s^-(x))$ is the right-shift of $\alpha(x)$ with a $0$ added to the left, and $\alpha(s^+(X))$ is the right-shift of $\alpha(x)$ with a $1$ added to the left.
It follows that the angles corresponding to the vertices of a directed cycle in $\dG_r(L)$ are all rational multiples of $\pi$ such that for any two vertices $x_1$ and $x_2$ of the cycle there exist $k_1,k_2\in\bN$ such that $2^{k_1}\alpha(x_1)-2^{k_2}\alpha(x_2)\in\bZ$.
In the extremal case where $\dG_r(L)$ has $\ell$ arcs, this property holds for any two vertices $x_1,x_2$ in the same connected component of $\dG_r(L)$.

\begin{proposition}\label{thm:suspensionupper}
For any suspension $S$ with $n$ points, $e_r(S)\leq  \lceil\frac{n^2}{4}+\frac{5n}{2}\rceil+2$.
\end{proposition}
\begin{proof}
Let $c:=\card{C}$.
Thus $n=\ell+c$.
We estimate $e_r(S)$ by decomposing it as follows:
\[e_r(S) = e_r(L,C) + e_r(L) + e_r(C,L) +e_r(C).\]

Since each point in $L$ is joined to each point in $C$, $e_r(L,C)=\ell c$.
As shown above, each point in $L$ has at most one in-neighbour in $L$, so $e_r(L)\leq\ell$.
A sphere intersects a line in at most two points, hence $e_r(C,L)\leq 2c$.
Furthermore, a sphere with centre on $\CC$ intersects $\CC$ in at most two points, which gives $e_r(C)\leq 2c$.
Therefore,
\begin{align*}
e_r(S) &\leq \ell c + \ell + 4c = (\ell+4)(c+1)-4\\
&\leq \left\lfloor\frac{(\ell+c+5)^2}{4}\right\rfloor - 4
= \left\lfloor\frac{(n+5)^2}{4}\right\rfloor - 4\\
&= \left\lfloor \frac{n^2+10n+9}{4} \right\rfloor
= \left\lceil \frac{n^2}{4} + \frac{5n}{2} \right\rceil + 2.\qedhere
\end{align*}
\end{proof}
For sufficiently large $n$, we can almost attain the bound in Proposition~\ref{thm:suspensionupper}.
\begin{theorem}\label{thm:suspension}
For each $n\geq 13$ there is a suspension $S$ with $n$ points such that $e_r(S)\geq \lceil\frac{n^2}{4}+\frac{5n}{2}\rceil+1$.
\end{theorem}
\begin{proof}
Consider the full binary tree consisting of $0$ and its successors obtained by repeatedly applying $s^-$ and $s^+$ (Figure~\ref{fig:tree}).
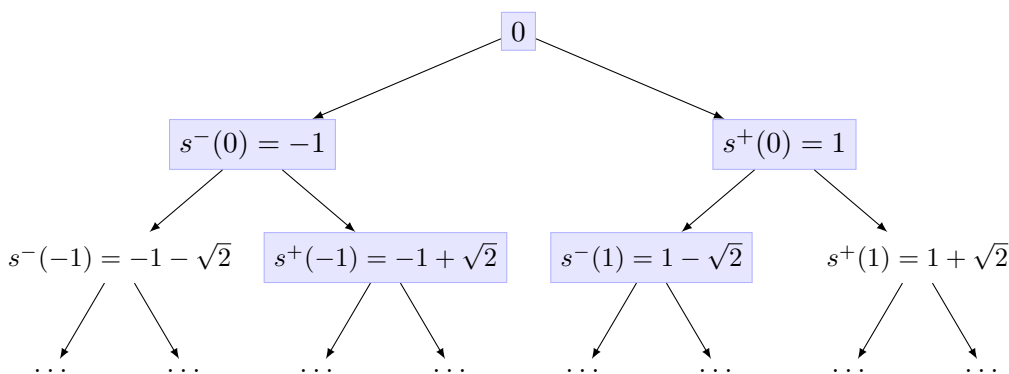
\begin{figure}
\centering
\begin{tikzpicture}
[level 1/.style={sibling distance=70mm,-latex},
 level 2/.style={sibling distance=35mm},
 level 3/.style={sibling distance=17.5mm},
 level 4/.style={sibling distance=10mm}]
  \node[rectangle,draw=blue!30,fill=blue!10] {$\,\!0$}
    child {node[rectangle,draw=blue!30,fill=blue!10] {$s^-(0)=-1$}
      child {node {\small $s^-(-1)=-1-\sqrt{2}$}
        child {node[rectangle] {\dots}}
        child {node {\dots}}
      }
      child {node[rectangle,draw=blue!30,fill=blue!10] {\small $s^+(-1)=-1+\sqrt{2}$}
        child {node {\dots}}
        child {node {\dots}}
      }
    }
    child {node[rectangle,draw=blue!30,fill=blue!10] {$s^+(0)=1$}
      child {node[rectangle,draw=blue!30,fill=blue!10] {\small $s^-(1)=1-\sqrt{2}$}
        child {node {\dots}}
        child {node {\dots}}
      }
      child {node {\small $s^+(1)=1+\sqrt{2}$}
        child {node {\dots}}
        child {node {\dots}}
      }
    };
\end{tikzpicture}
\caption{The full binary tree with root $0\in\CL$}\label{fig:tree}
\end{figure}
Let $L$ be the vertex set of any subtree with $\ell:=\lfloor (n-3)/2\rfloor\geq 5$ vertices that contains at least the $5$ points $0$, $\pm 1$ and $\pm(\sqrt{2}-1)$.
We then have $e_r(L)=\ell-1$.
We next show that we can choose $c:=\lceil(n+3)/2\rceil\geq 8$ points on $\CC$ such that $e_r(C,S)=4c$.

For the vertices of $C$, choose the vertices of $\lfloor c/4\rfloor\geq 2$ squares inscribed in $\CC$.
For each of these vertices $a$, define $r(a)=\sqrt{2}$.
Since $\pm 1\in L$, each $a$ has out-degree $4$.
If $c$ is not divisible by $4$, choose two of the squares to be such that one vertex $a$ of one square and one vertex $b$ of the other square are at distance $\sqrt{8(\sqrt{2}-1)}$.
This distance is chosen so that there exists a point $p$ on $\CC$ such that $d(a,p)=d(b,p)=\sqrt{4-2\sqrt{2}}=d(p,q)$, where $q$ is the point $\sqrt{2}-1\in L$.
See Figure~\ref{fig:circlesquare}.
\begin{figure}
\centering
\begin{tikzpicture}[line cap=round,line join=round,>=triangle 45,scale=0.6]
\draw(1.510597410548229,5.851976395257644) -- (-0.6319763952576445,0.13059741054822893) -- (5.089402589451771,-2.0119763952576446) -- (7.231976395257644,3.709402589451771) -- cycle;
\draw(1.510597410548229,-2.0119763952576446) -- (-0.6319763952576447,3.7094025894517704) -- (5.089402589451771,5.851976395257644) -- (7.231976395257642,0.1305974105482295) -- cycle;
\draw(1.934861479260157,1.92) -- (1.934861479260157,2.344264068711928) -- (1.510597410548229,2.344264068711928) -- (1.510597410548229,1.92) -- cycle; 
\draw(3.3,1.92) circle (4.32cm);
\draw (1.510597410548229,5.851976395257644)-- (-0.6319763952576445,0.13059741054822893);
\draw (-0.6319763952576445,0.13059741054822893)-- (5.089402589451771,-2.0119763952576446);
\draw (5.089402589451771,-2.0119763952576446)-- (7.231976395257644,3.709402589451771);
\draw (7.231976395257644,3.709402589451771)-- (1.510597410548229,5.851976395257644);
\draw (1.510597410548229,-2.0119763952576446)-- (-0.6319763952576447,3.7094025894517704);
\draw (-0.6319763952576447,3.7094025894517704)-- (5.089402589451771,5.851976395257644);
\draw (5.089402589451771,5.851976395257644)-- (7.231976395257642,0.1305974105482295);
\draw (7.231976395257642,0.1305974105482295)-- (1.510597410548229,-2.0119763952576446);
\draw (1.510597410548229,-2.0119763952576446)-- (1.510597410548229,5.851976395257644);
\draw (-1.02,1.92)-- (3.3,1.92);
\draw [fill=black] (3.3,1.92) circle (1.5pt) node[right] {$0$};
\draw [fill=black] (-1.02,1.92) circle (1.5pt);
\draw[color=black] (-1.4,2) node {$p$};
\draw [fill=black] (1.510597410548229,1.92) circle (1.5pt);
\draw [fill=black] (1.510597410548229,5.851976395257644) circle (1.5pt);
\draw[color=black] (1.3,6.14) node {$a$};
\draw [fill=black] (1.510597410548229,-2.0119763952576446) circle (1.5pt);
\draw[color=black] (1.3,-2.3) node {$b$};
\draw[color=black] (6.2,5.8) node {$\CC$};
\draw[color=black] (2.8,3.8) node {$\scriptstyle\sqrt{2(\sqrt{2}-1)}$};
\draw[color=black] (2.4,1.5) node {$\scriptstyle\sqrt{2}-1$};
\draw [fill=black] (7.231976395257644,3.709402589451771) circle (1.5pt);
\draw [fill=black] (-0.6319763952576445,0.13059741054822893) circle (1.5pt);
\draw [fill=black] (5.089402589451771,-2.0119763952576446) circle (1.5pt);
\draw [fill=black] (-0.6319763952576447,3.7094025894517704) circle (1.5pt);
\draw [fill=black] (7.231976395257642,0.1305974105482295) circle (1.5pt);
\draw [fill=black] (5.089402589451771,5.851976395257644) circle (1.5pt);
\end{tikzpicture}
\caption{Construction of points on $\CC$}\label{fig:circlesquare}
\end{figure}
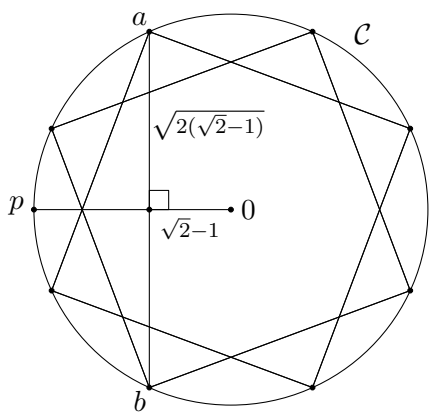
The same holds for any of the other three vertices of the square with vertex $p$ inscribed in $\CC$.
We then add $c-4\lfloor c/4\rfloor$ vertices of this square to $C$ and define $r(p)= \sqrt{4-2\sqrt{2}}$ for each of these vertices $p$, to obtain a set $C$ of exactly $c$ points such that $e_r(C,S)=4c$.
Then
\begin{align*}
e_r(S) &= e_r(L,C) + e_r(L) + e_r(C,S)\\
&= \ell c+\ell - 1 + 4c = (\ell+4)(c+1)-5 \\
&= \left\lfloor\frac{(n+5)^2}{4}\right\rfloor - 5 = \left\lceil\frac{n^2}{4}+\frac{5n}{2}\right\rceil+1. \qedhere
\end{align*}
\end{proof}

If the upper bound of Proposition~\ref{thm:suspensionupper} is attained, it would have to be because of a very special algebraic coincidence, and we believe that this is not possible.
The following observation may help to prove this.
For a point $x\in C$ let $\fhi(x)$ be the angle $\myangle x0y$, where $y\in\CC$ satisfies $d(x,y)=r(x)$.
Then $r(x)=2\sin\frac{\fhi(x)}{2}$.
If also $(x,z)\in\dE_r(S)$ where $z\in L$, then it follows from $r(x)=d(x,z)=\sqrt{1+z^2}$ by elementary trigonometrical relations that $\sin\theta(z) \sin\frac{\fhi(x)}{2} = 1/2$.
It follows from a result of M.~Newman \cite{Newman} that there are only two solutions to this equation with both $\theta(z)$ and $\fhi(x)$ rational multiples of $\pi$ in the interval $(0,\pi)$, namely $(\theta(z),\fhi(x))=(\pi/4,\pi/2)$ and $(\theta(z),\fhi(x))=(\pi/2,\pi/3)$.
The first solution corresponds to our construction in Theorem~\ref{thm:suspension}, while the second solution correspond to an analogous construction with inscribed regular hexagons, which turns out to be worse than our construction.

\section{A simple upper bound}\label{section:simple}
We will use the K\H{o}vari-S\'os-Tur\'an Theorem in the standard form and in a form for directed graphs.
See for instance \cite[Theorem~2.2]{Bollobas}.
\begin{KST1}
Let $G=(V_1\cup V_2,E)$ be a bipartite graph with parts of sizes $\card{V_1}=m$ and $\card{V_2}=n$.
Suppose that $G$ does not contain a complete bipartite graph with parts of size $r$ and $s$, with the part of size $r$ contained in $V_1$ and the part of size $s$ contained in $V_2$.
Then $\card{E}\leq (s-1)^{1/r}(m-r+1)n^{1-1/r}+(r-1)n$.
\end{KST1}
Denote by $\dK_{r,s}$ the digraph $(V,\dE)$ with $V=\set{a_1,\dots,a_r,b_1,\dots,b_s}$ and \[\dE=\setbuilder{(a_i,b_j)}{i=1,\dots,r; j=1,\dots,s}.\]
\begin{KST2}
Let $\dG=(V,\dE)$ be a directed graph with $\card{V}=n$ vertices.
Suppose that $G$ does not contain a copy of $\dK_{r,s}$.
Then $\card{\dE}\leq (s-1)^{1/r}n^{2-1/r}+(r-1)n$.
\end{KST2}
\begin{proof}
Apply the K\H{o}vari-S\'os-Tur\'an Theorem I to the bipartite double cover of~$G$.
\end{proof}

The following upper bound improves the error term from \cite{AEP}.
The proof uses an induction argument which, although conceptually simple, needs some computation.
\begin{theorem}\label{thm:induction}
There exists $A>0$ such that for all $n\geq 0$ and for any set $S$ of $n$ points in $\bR^3$ and any function $r\colon S\to(0,\infty)$, $e_r(S)\leq \frac{n^2}{4}+An^{5/3}$.
\end{theorem}
\begin{proof}
Fix a constant $p\geq 3$, to be determined later.
The K\H{o}vari--S\'os--Tur\'an Theorem II guarantees the existence of a $n_0=n_0(p)$ such that any directed graph on $n\geq n_0$ vertices and with at least $0.23n^2$ edges contains a $\dK_{p,p}$.
We can ensure that the theorem holds for all $n\leq n_0$ by taking $A=A(n_0)$ large enough.
Assume next that $n>n_0$ and that the theorem holds for sets of up to $n-1$ points.

Without loss of generality, $e_r(S)\geq 0.23n^2$.
Let $L$ and $C$ be the two classes of a $\dK_{\ell,c}$ contained in $\dG_r(S)$ chosen such that $\ell+c$ is maximal among all such $\dK_{\ell,c}$, where $\ell:=\card{L}$, $c:=\card{C}$ and $\ell,c\geq p$.
Then $C$ lies on a circle $\CC$ and $L$ on its axis of symmetry $\CL$, and $r(p)=d(p,C)$ for each $p\in L$.
That is, $C\cup L$ forms a suspension of maximum cardinality among all those contained in $S$ with $\card{C},\card{L}\geq p$.
Let $T=S\setminus(C\cup L)$ and write $t:=\card{T}$.
We will bound $e_r(S)$ by writing it as the following sum and then bounding each term separately.
\begin{align*}
e_r(S)  &= e_r(L,C) + e_r(C,L)+ e_r(T,L)+e_r(L)+e_r(C)\\
&\quad +e_r(L,T)+e_r(T,C)+e_r(C,T)+e_r(T).
\end{align*}
See Figure~\ref{fig1}.
\begin{figure}
\centering
\begin{tikzpicture}[-triangle 45, bend angle=15, bend left, outer sep=0pt, scale=3, auto, shorten >=3pt, shorten <=3pt, scale=0.75]
\node (L) at (-0.5,1) [minimum size =25mm, circle, draw, label=above left:$L$] {$\leq\ell$};
\node (C) at (2.5,1) [minimum size =25mm, circle, draw, label=above right:$C$] {$\leq 2c$};
\node (T) at (1,-0.5) [minimum size =25mm, circle, draw, label=-20:$T$] {$\leq\frac{t^2}{4}+At^{5/3}$};
\draw (L) to node [pos=0.53] {$\ell c$} (C);
\draw (C) to node [pos=0.6,swap] {$\leq 2 c$} (L);
\draw (L) to node [pos=0.6] {$\leq t$} (T);
\draw (T) to node [pos=0.6] {$\leq 2 t$} (L);
\draw (C) to node [outer sep=-4pt,pos=0.6] {$\leq 2c+2^{1/3}tc^{2/3}$} (T);
\draw (T) to node [outer sep=-2pt,pos=0.4] {$\leq 2 t$} (C);
\end{tikzpicture}
\caption{Bounding $e_r(S)$ from above}\label{fig1}
\end{figure}
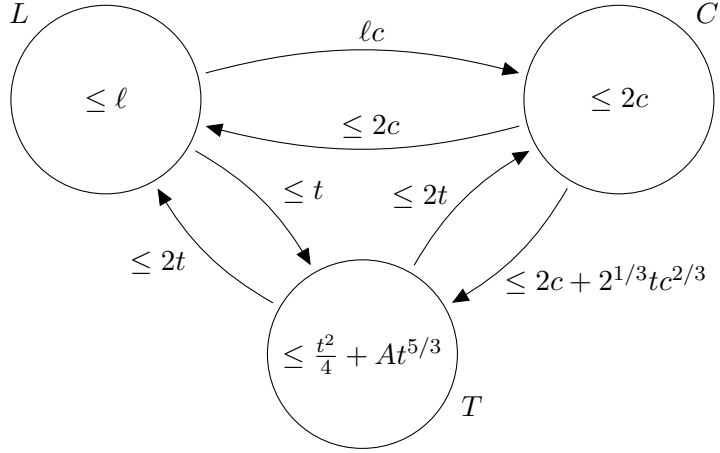
Since each vertex on $L$ is joined to each vertex on $C$, we have $e_r(L,C)=\ell c$.
A sphere and a line intersects in at most $2$ points, hence $e_r(C,L)\leq 2c$ and $e_r(T,L)\leq 2t$.
Each vertex on $L$ has at most one in-neighbour on $L$, hence $e_r(L)\leq\ell$.
A circle and a sphere with centre on the circle intersect in at most $2$ points, hence $e_r(C)\leq 2c$.

Suppose that some $x\in T$ has at least two in-neighbours $y_1,y_2\in L$, say.
It then follows that $x$ lies on the intersection of the spheres with centre $y_i$ and radius $r(y_i)=r(y_i,\CC)$, $i=1,2$, which is the circle $\CC$.
This contradicts the maximality of $C\cup L$.
Therefore, $e_r(L,T)\leq t$.

Suppose that some $x\in T$ has at least three out-neighbours $y_1,y_2,y_3\in C$, say.
Then necessarily $x\in\CL$, which contradicts the maximality of $C\cup L$.
Therefore, $e_r(T,C)\leq 2t$.

There is no $\dK_{3,3}$ from $C$ to $T$, otherwise there would be a $\dK_{3,3,3}$ from $L$ to $C$ to $T$, which is not realisable in $\bR^3$.
By the K\H{o}vari-S\'os-Tur\'an Theorem I (applied to the reverse bipartite graph from $T$ to $C$), $e_r(C,T)\leq 2c+ 2^{1/3}tc^{2/3}$.

Finally, we estimate $e_r(T)\leq t^2/4+At^{5/3}$ by the induction hypothesis:
\begin{align*}
e_r(S) &\leq \ell c + 2c + 2t +\ell+2c+t+2t+(2c+2^{1/3}tc^{2/3})+\frac{t^2}{4}+At^{5/3}\\
&= (\ell+6)(c+1) - 6 + \frac{t^2}{4} + At^{5/3} +5t + 2^{1/3} tc^{2/3}\\
&= \left(\frac{\ell+c+7}{2}\right)^2 - \left(\frac{\ell-c+5}{2}\right)^2 + \frac{t^2}{4} + At^{5/3} +5t -6 + 2^{1/3} tc^{2/3}\\
&\leq \left(\frac{n-t+7}{2}\right)^2 + \frac{t^2}{4} + At^{5/3} +5t -6 + 2^{1/3} t(n-t)^{2/3}\\
&= \frac14(n^2-2nt+2t^2+14n+6t+25+4At^{5/3}) + 2^{1/3}t(n-t)^{2/3}.
\end{align*}
Since $\ell,c\geq p\geq 3$,
\[14n+6t+25\leq 14n+6(n-2p)+25 = 20n-12p+25 < 20n,\]
hence,
\[ e_r(S)\leq \frac14(n^2-2nt+2t^2+20n+4At^{5/3}) + 2^{1/3}t(n-t)^{2/3},\]
which will be $\leq \frac14 n^2 + An^{5/3}$, thus finishing the induction step, if
\begin{equation}\label{one}
-2nt + 2t^2 + 20n +4At^{5/3} + 4\cdot 2^{1/3}t(n-t)^{2/3} \leq 4An^{5/3}.
\end{equation}
We next show that \eqref{one} holds if $n-t$ is sufficiently large.
Since $n-t\geq 2p$, we can ensure that $n-t$ is large by choosing $p$ large.

First suppose that $t\geq n/2$.
The inequality \eqref{one} will follow if
\[ -2nt+2t^2+20n+4\cdot 2^{1/3}t(n-t)^{2/3}\leq 0,\]
which is equivalent to 
\[ (n-t)^{2/3}((n-t)^{1/3}-2^{4/3})\geq 10n/t.\]
Since $10n/t\leq 20$, it is sufficient to have
\[ (n-t)^{2/3}((n-t)^{1/3}-2^{4/3})\geq 20,\]
which holds if $n-t\geq 58$.
We can ensure this by requiring that $p\geq 29$.

Next consider the remaining case where $t<n/2$.
Since $t\mapsto t(n-t)^{2/3}$ is increasing on $[0,n/2]$, we have $t(n-t)^{2/3}\leq(n/2)^{5/3}$.
Also, $-2nt+2t^2\leq 0$ and $20n\leq 20n^{5/3}$.
Therefore, to derive \eqref{one}, it is sufficient to show that 
\[ 20n^{5/3} + 4A(n/2)^{5/3} +4\cdot 2^{1/3}(n/2)^{5/3} \leq 4An^{5/3}.\]
This inequality is equivalent to
\[ A \geq \frac{5+2^{-4/3}}{1-2^{-5/3}},\]
which can be ensured.
This finishes the induction step.
\end{proof}

\begin{corollary}\label{cor:induction}
There exists $p\geq 3$, $A>0$ and $n_0\in\bN$ such that any set $S\subset\bR^3$ of $n\geq n_0$ points with $e_r(S)\geq 0.23n^2$ can be partitioned into a suspension $C\cup L$ and a remainder set $T$ such that $\card{C}=c\geq p$, $\card{L}=\ell\geq p$, $\card{T}=t$ and
\begin{align}
&\mathrel{\phantom{\leq}} e_r(S)+\left(\frac{\ell-c+5}{2}\right)^2\notag \\
&\leq \frac14(n^2-2nt+2t^2+14n+6t+25+4At^{5/3})+2^{1/3}t(n-t)^{2/3}\label{bound}\\
&\leq \frac14 n^2+An^{5/3}. \notag
\end{align}
\end{corollary}
\begin{proof}
Proceed as in the proof of Theorem~\ref{thm:induction}, but instead of applying an induction hypothesis to $T$, apply Theorem~\ref{thm:induction} itself to $T$.
\end{proof}

\section{Proof of the main theorems}\label{section:mainproofs}
\begin{proof}[Proof of Theorem~\ref{thm:3dstability}]
Without loss of generality, by requiring $\delta\leq 1/50$, we may assume that $e_r(S) > 0.23 n^2$.
By Corollary~\ref{cor:induction} there exist $p\geq 3$, $n_0\in\bN$, $A>0$ such that any set $S\subset\bR^3$ of at least $n\geq n_0$ points can be partitioned into a suspension $C\cup L$ with cardinalities $\card{C}=c\geq p$ and $\card{L}=\ell\geq p$ and a remainder set $T$ of cardinality $\card{T}=t$ such that \eqref{bound} holds.
Combine \eqref{bound} with the lower bound $e_r(S)\geq (1/4-\delta)n^2$ to obtain
\begin{align*}
2nt-2t^2 &\leq 4\delta n^2 + 14n + 6t + 25 + 4At^{5/3} + 4\cdot 2^{1/3}t(n-t)^{2/3}\\
& < 5\delta n^2
\end{align*}
for $n$ sufficiently large depending on $\delta$ and $A$.
The resulting quadratic inequality in $t$ implies that
\[ t < \left(\frac{1-\sqrt{1-10\delta}}{2}\right)n\quad\text{or}\quad t > \left(\frac{1+\sqrt{1-10\delta}}{2}\right)n.\]
Thus, either $t$ or $n-t$ is small.
We next show that $n-t$ is not small.
Recall from the proof of Theorem~\ref{thm:induction} that the suspension $C\cup L$ was chosen such that $c+\ell$ is maximised subject to $\ell,c\geq p$.
This implies that the favourite distance digraph $\dG_r(S)$ does not contain a $\dK_{p,n-t}$, hence by the K\H{o}vari-S\'os-Tur\'an Theorem II,
\[ e_r(S) \leq (n-t-1)^{1/p} n^{2-1/p} + (p-1)n.\]
Combine this with $e_r(S)\geq(1/4-\delta)n^2$ to obtain 
\[ t+1 \leq \left(1-\left(\frac{1}{4}-\delta-\frac{p-1}{n}\right)^p\right)n.\]
If $t > (1+\sqrt{1-10\delta})n/2$, then
\[ \frac{1+\sqrt{1-10\delta}}{2} < 1-\left(\frac{1}{4}-\delta-\frac{p-1}{n}\right)^p,\]
which gives a contradiction if $\delta$ is sufficiently small and $n$ sufficiently large, both depending on $p$.
Therefore,
\[ t < \left(\frac{1-\sqrt{1-10\delta}}{2}\right)n <\epsi n\]
if $\delta$ is sufficiently small depending on $\epsi$.

The bound
\[ e_r(S) + \left(\frac{\ell-c+5}{2}\right)^2 \leq \frac{n^2}{4} + An^{5/3}\]
from Corollary~\ref{cor:induction}, together with the lower bound $e_r(S)\geq (1/4-\delta)n^2$ gives
\[ \left(\frac{\ell-c+5}{2}\right)^2 \leq \delta n^2 + An^{5/3} < 2\delta n^2\]
for $n$ sufficiently large depending on $\delta$ and $A$.
Then $\abs{\ell-c}\leq 3\sqrt{\delta} n$ for $n$ sufficiently large depending on $\delta$, hence $\abs{\ell-c} <\epsi n$ for $\delta$ sufficiently small depending on $\epsi$.
It follows that $\abs{\ell-n/2}<\epsi n$ and $\abs{c-n/2}<\epsi n$.
\end{proof}

\begin{proof}[Proof of Theorem~\ref{thm:3dextremal}]
Let $S\subset\bR^3$ with $\card{S}=n$ and $r\colon S\to(0,\infty)$ satisfy $e_r(S)=f_3(n)$.
Since $f_3(n)\geq n^2/4 + 5n/2$ by Theorem~\ref{thm:suspension}, Corollary~\ref{cor:induction} gives that for each $\epsi>0$ there exists $A>0$ and $n_0\in\bN$ such that for all $n\geq n_0$, $S$ consists of a spindle $C\cup L$ except for an exceptional set $T\subseteq S$
such that
\[
e_r(S) \leq \frac14(n^2-2nt+2t^2+14n+6t+25+4At^{5/3}) + 2^{1/3}t(n-t)^{2/3}.
\]
By Theorem~\ref{thm:3dstability} we can also assume that $t\leq n/4$ by making $n$ sufficiently large.
It follows that
\[ \frac{n^2}{4} + \frac{5n}{2} \leq \frac14(n^2-2nt+2t^2+14n+6t+25+4At^{5/3}) + 2^{1/3}t(n-t)^{2/3},\]
hence 
\begin{align*}
0 & \leq 2n(2-t)+2t^2+6t+25+4At^{5/3} + 4\cdot2^{1/3}t(n-t)^{2/3}\\
&\leq 2n(2-t)+2t^2+Btn^{2/3}:= g(t)
\end{align*}
for sufficiently large $B$ depending on $A$.
It is easy to check that for each sufficiently large fixed $n$, $g(t)$ is decreasing on $[0,n/4]$ and negative for $t=3$.
It follows that $t\leq 2$.
\end{proof}

\begin{proof}[Proof of Theorem~\ref{thm:3d}]
The lower bound is given by Theorem~\ref{thm:suspension}.
Consider an extremal example with sufficiently many points so that Theorem~\ref{thm:3dextremal} can be applied.

If $T$ is empty, then $S$ is a suspension, and Proposition~\ref{thm:suspensionupper} provides the upper bound.

If $\card{T}=1$, then $\ell+c=n-1$ and 
\begin{align*}
e_r(S) &= e_r(L,C) + e_r(C,L) + e_r(C) + e_r(L) +  e_r(L,T) + e_r(C,T) + e_r(T,L\cup C)\\
&\leq \ell c + 2c + 2c + \ell + 1 + c + 4 = (\ell+5)(c+1)\\
&\leq \left\lfloor \left(\frac{\ell+c+6}{2}\right)^2\right\rfloor = \left\lceil \frac{n^2}{4}+\frac{5n}{2}\right\rceil + 6,
\end{align*}
where the separate estimates are made as in the proof of Theorem~\ref{thm:induction}.

Similarly, if $\card{T}=2$, then $\ell+c=n-2$ and 
\begin{align*}
e_r(S) &= e_r(L,C) + e_r(C,L) + e_r(C) + e_r(L)\\
\phantom{e_r(S)} &\mathrel{\phantom{=}} +  e_r(L,T) + e_r(C,T) + e_r(T,L\cup C) + e_r(T)\\
&\leq \ell c + 2c + 2c + \ell + 2 + 2c + 2\cdot4 + 2 = (\ell+6)(c+1) + 6\\
&\leq \left\lfloor \left(\frac{\ell+c+7}{2}\right)^2\right\rfloor+6 = \left\lceil \frac{n^2}{4}+\frac{5n}{2}\right\rceil + 12. \qedhere
\end{align*}
\end{proof}

\end{document}